\documentclass[12pt,reqno]{amsart}
\usepackage{eurosym}
\usepackage{amsmath}
\usepackage{amsfonts}
\usepackage[asymmetric]{geometry}
\usepackage{hyperref}

\newtheorem{theorem}{Theorem}[section]

\newtheorem{corollary}{Corollary}
\numberwithin{corollary}{section}
 \theoremstyle{definition}

\theoremstyle{remark}
\newtheorem{remark}[theorem]{Remark}
\numberwithin{equation}{section}

\begin{document}
\title{Weighted Hardy Type inequalities with Robin Boundary Conditions}
\author{Ismail Kombe and Abdullah Yener}
\address{\href{http://www.iticu.edu.tr/Akademi/Akademisyen/ikombe}{Ismail
Kombe}, Department of Mechatronics Engineering\\
Faculty of Engineering\\
\href{https://ticaret.edu.tr/}{\.{I}stanbul Commerce University }\\
K\"{u}\c{c}\"{u}kyal\i , 34840, \.{I}stanbul, T\"{u}rkiye }
\email{ikombe@ticaret.edu.tr}
\address{\href{http://tubis.ticaret.edu.tr/_Adek/CV/CV.aspx?adi=si6m5iTB2e55CVeeVSnj8SKmEO3iz1HNYf8LNmziaiTqf++bfJxnc678bZZuRnZt%
}{Abdullah Yener}, \href{https://ticaret.edu.tr/matematik/}{Department of
Mathematics}\\
\href{https://ticaret.edu.tr/itbf/}{Faculty of Humanities and Social Sciences%
}\\
\href{https://ticaret.edu.tr/}{\.{I}stanbul Commerce University }\\
Beyo\u{g}lu, 34445, \.{I}stanbul, T\"{u}rkiye}
\email{ayener@ticaret.edu.tr }
\date{August 8, 2022}
\dedicatory{}

\begin{abstract}
In this paper we establish a general weighted Hardy type inequality for the $%
p-$Laplace operator with Robin boundary condition. We provide various
concrete examples to illustrate our results for different weights.
Furthermore, we present some Heisenberg-Pauli-Weyl type inequalities with
boundary terms on balls with radius $R$ at the origin in $\mathbb{R}^n$.
\end{abstract}

\maketitle

%\subjclass[2010]{ 35H10, 26D10, 46E35.}

\section{Introduction\label{int}}

The classical Hardy inequality states that 
\begin{equation}  \label{1.1}
\int_{%
%TCIMACRO{\U{211d} }%
%BeginExpansion
\mathbb{R}
%EndExpansion
^{n}}\left\vert \nabla \phi \left( x\right) \right\vert ^{p}dx\geq
\left\vert \frac{n-p}{p}\right\vert ^{p}\int_{%
%TCIMACRO{\U{211d} }%
%BeginExpansion
\mathbb{R}
%EndExpansion
^{n}}\frac{\left\vert \phi \left( x\right) \right\vert ^{p}}{\left\vert
x\right\vert ^{p}}dx,
\end{equation}%
and holds for all $\phi \in C_{0}^{\infty }\left( 
%TCIMACRO{\U{211d} }%
%BeginExpansion
\mathbb{R}
%EndExpansion
^{n}\right) $ if $1<p<n$, and for all $\phi \in C_{0}^{\infty }\left( 
%TCIMACRO{\U{211d} }%
%BeginExpansion
\mathbb{R}
%EndExpansion
^{n}\backslash \left\{ 0\right\} \right) $ if $p>n$. The constant $%
|(n-p)/p|^{p}$ in the right-hand side of $\ref{1.1}$ is sharp but not
achieved. This inequality plays an important role in many areas such as
analysis, geometry and mathematical physics. Owing to this, there is a vast
amount of literature on the Hardy type inequalities together with their
variations and generalizations.

For instance, V.G. Maz'ya \cite{Mazya} proved the following weighted
integral inequality 
\begin{equation}  \label{1.2}
\int_{\mathbb{R}^n} |x_n|^{p-1} |\nabla \phi\left(x\right)|^p dx \geq \frac{1%
}{(2p)^p}\int_{\mathbb{R}^n} \frac{|\phi\left(x\right)|^p}{%
\left(x_{n-1}^2+x_n^2\right)^{\frac{1}{2}}} dx,
\end{equation}
where $\phi \in C_0^{\infty}(\mathbb{R}^n)$.

In the paper \cite{Davies}, E.B. Davies proved the following extension of
Hardy's inequality for convex domains $\Omega \subset 
%TCIMACRO{\U{211d} }%
%BeginExpansion
\mathbb{R}
%EndExpansion
^{n},$ $n\geq 2,$%
\begin{equation}  \label{1.3}
\int_{\Omega }\left\vert \nabla \phi\left(x\right) \right\vert^{2}dx\geq 
\frac{1}{4}\int_{\Omega }\frac{\left\vert\phi\left(x\right)\right\vert^{2}}{%
\delta ^{2}\left( x\right) }dx,
\end{equation}%
where $\phi \in C_{0}^{\infty }\left( \Omega \right)$. Here

\begin{equation}  \label{1.4}
\delta\left(x\right) = \mathrm{dist}\, \left(x,\partial\Omega\right) =
\min_{y\in\partial\Omega} \left\vert x-y\right\vert
\end{equation}
is the distance function to the boundary $\partial \Omega$. Moreover the
constant $\frac{1}{4}$ is sharp and not achieved. We refer the interested
reader to the monographs \cite{Kufner}, \cite{Davies}, \cite{Lewis}, \cite%
{Moradifam}, \cite{Ruzhansky} and the paper \cite{Goldstein-Kombe-Yener}.

On the other hand, very little work has been done on Hardy type inequalities
with Robin boundary conditions. However, there has been some initiation in
this area of interest. In an interesting paper, H. Kova\v{r}ik and A. Laptev 
\cite{Kovarik-Laptev} proved, among other results, the following Hardy
inequality for Laplace operators with Robin boundary conditions, 
\begin{equation}  \label{1.5}
\begin{aligned} \int_{\Omega }\left\vert \nabla \phi \right\vert
^{2}dx+\int_{\partial \Omega }\sigma \left\vert \phi \right\vert ^{2}d\nu
\geq \frac{1}{4}& \int_{\Omega
}\left(\delta\left(x\right)+\frac{1}{2\sigma\left(p\left(x\right)\right)}%
\right)^{-2}\left\vert\phi\right\vert^2dx\\ +& \frac{1}{4}\int_{\Omega
}\left(R_{in}+\frac{1}{2\sigma\left(p\left(x\right)\right)}\right)^{-2}\left%
\vert\phi\right\vert^2dx\\ + &\frac{1}{2}\int_{\partial \Omega
}\left(R_{in}+\frac{1}{2\sigma\left(y\right)}\right)^{-1}\left
\vert\phi\right\vert^2d\nu, \end{aligned}
\end{equation}%
where $\Omega\subset \mathbb{R}^n$ is a bounded convex domain with smooth
boundary $\partial\Omega$, $\phi\in H^1(\Omega)$, $0\le \sigma \in
L^{\infty}\left(\partial\Omega\right)$, $d\nu $ is the surface measure on $%
\partial \Omega $ and $R_{in}:=\text{sup}\{\delta(x): x\in\Omega\}$. Here $%
p\left(x\right)$ is the projection function $p:\Omega\setminus S \to
\partial \Omega$ defined by $\label{projection} p\left(x\right) := y\in
\partial\Omega\ : \, \delta\left(x\right) = |x-y|, \quad x\in
\Omega\setminus S$, where $S\subset\Omega$ the subset of points in $\Omega$
for which there exist at least two points $y_1, y_2 \in \partial\Omega$ so
that the minimum in \eqref{1.4} is achieved.

Later, T. Ekholm, H. Kova\v{r}ik and A. Laptev \cite{Ekholm-Kovarik-Laptev}
studied the best constant in a Hardy inequality for the $p-$Laplace operator
on convex domains with Robin boundary conditions. Moreover, there has also
been an interest regarding Hardy type inequalities for functions which does
not vanish on the boundary of a given domain, see e.g. \cite{Adimurthi1}, 
\cite{Bercio-Cassani-Gazzola}, \cite{Cowan}, \cite{Wang-Zhu} and \cite%
{Pinchover-Versano}.

In light of the results mentioned above, it is natural to investigate under
what conditions the following general weighted Hardy-type inequalities with
boundary terms are valid

\begin{eqnarray}
\int_{\Omega }a\left( x\right) \left\vert \nabla \phi \right\vert ^{p}dx
&\geq &\int_{\Omega }b\left( x\right) \left\vert \phi \right\vert ^{p}dx
+\int_{\partial \Omega }\beta \left( x\right) \left\vert \phi \right\vert
^{p}d\nu.
\end{eqnarray}

Within this frame work, the main objective of this article is to study the
general weighted Hardy type inequalities for the $p-$Laplace operators with
Robin boundary conditions. We should emphasize that our unifying method is
quite practical and constructive to obtain several weighted Hardy, Maz'ya
and Heisenberg-Pauli-Weyl type inequalities with boundary terms.

\section{Weighted Hardy Type Inequalities With Boundary Terms \label%
{section2}}

In what follows, $\Omega $ represents smooth bounded domains contained in $%
%TCIMACRO{\U{211d} }%
%BeginExpansion
\mathbb{R}
%EndExpansion
^{n},$ $\partial _{\nu }u$ denotes the exterior normal derivative of $u$ at
the boundary of $\Omega $ and $d\nu $ is the surface measure on $\partial
\Omega $. The generic point is $x=\left( x_{1},\ldots ,x_{n}\right) \in 
%TCIMACRO{\U{211d} }%
%BeginExpansion
\mathbb{R}
%EndExpansion
^{n}$ and $r=\left\vert x\right\vert =\sqrt{x_{1}^{2}+\cdots +x_{n}^{2}}.$
We denote by $B_{R}=\left\{ x\in 
%TCIMACRO{\U{211d} }%
%BeginExpansion
\mathbb{R}
%EndExpansion
^{n}:\left\vert x\right\vert <R\right\} $ the open ball at the origin with
radius $R$ in $\mathbb{R}^n$. The boundary is the sphere $\partial
B_{R}=\left\{ x\in 
%TCIMACRO{\U{211d} }%
%BeginExpansion
\mathbb{R}
%EndExpansion
^{n}:\left\vert x\right\vert =R\right\}$.

We are now ready to state the main result of this paper.

\begin{theorem}
\label{thm2.1}Let $\Omega $ be a smooth bounded domain contained in $%
%TCIMACRO{\U{211d} }%
%BeginExpansion
\mathbb{R}
%EndExpansion
^{n}.$ Let $a\in C^{1}\left( \Omega \right) $ be a nonnegative function, $%
u\in C^{\infty }\left( \Omega \right) $ be a positive function, $b\in L_{%
\text{loc}}^{1}\left( \Omega \right) $ and $\beta \in L_{\text{loc}%
}^{1}\left( \partial \Omega \right) $ such that%
\begin{equation}
-\nabla \cdot \left( a\left( x\right) \left\vert \nabla u\right\vert
^{p-2}\nabla u\right) \geq b\left( x\right) u^{p-1}\qquad \text{a.e. in }%
\Omega   \label{de1}
\end{equation}%
and%
\begin{equation}
a\left( x\right) \left\vert \nabla u\right\vert ^{p-2}\partial _{\nu
}u=\beta \left( x\right) u^{p-1}\qquad \text{a.e. at }\partial \Omega ,
\label{de2}
\end{equation}%
where $\partial _{\nu }u$ denotes the exterior normal derivative of $u$ at
the boundary of $\Omega .$ There exists a positive constant $c_{p}=c\left(
p\right) $ such that; if $p\geq 2,$ then%
\begin{eqnarray}
\int_{\Omega }a\left( x\right) \left\vert \nabla \phi \right\vert ^{p}dx
&\geq &\int_{\Omega }b\left( x\right) \left\vert \phi \right\vert
^{p}dx+c_{p}\int_{\Omega }a\left( x\right) \Big\vert\nabla \Big(\frac{\phi }{%
u}\Big)\Big\vert^{p}u^{p}dx  \notag \\
&&+\int_{\partial \Omega }\beta \left( x\right) \left\vert \phi \right\vert
^{p}d\nu   \label{e2}
\end{eqnarray}%
and if $1<p<2,$ then%
\begin{eqnarray}
\int_{\Omega }a\left( x\right) \left\vert \nabla \phi \right\vert ^{p}dx
&\geq &\int_{\Omega }b\left( x\right) \left\vert \phi \right\vert
^{p}dx+c_{p}\int_{\Omega }a\left( x\right) \frac{\left\vert \nabla \left( 
\frac{\phi }{u}\right) \right\vert ^{2}u^{2}}{\left( \left\vert \frac{\phi }{%
u}\nabla u\right\vert +\left\vert \nabla \left( \frac{\phi }{u}\right)
\right\vert u\right) ^{2-p}}dx  \notag \\
&&+\int_{\partial \Omega }\beta \left( x\right) \left\vert \phi \right\vert
^{p}d\nu   \label{e3}
\end{eqnarray}%
holds for all $\phi \in C^{\infty }\left( \Omega \right) .$ Here, $d\nu $ is
the surface measure on $\partial \Omega $.
\end{theorem}

\begin{proof}
For any $\phi \in C^{\infty }\left( \Omega \right) $ we now set $\varphi :=%
\frac{\phi }{u},$ where $0<u\in C^{\infty }\left( \Omega \right) .$ Thus,%
\begin{equation*}
\left\vert \nabla \phi \right\vert ^{p}=\left\vert \varphi \nabla u+u\nabla
\varphi \right\vert ^{p}.
\end{equation*}%
We now use the following convexity inequality:%
\begin{equation*}
\left\vert \xi +\eta \right\vert ^{p}\geq |\xi |^{p}+p|\xi |^{p-2}\xi \cdot
\eta +c_{p}|\eta |^{p},\qquad \forall \xi ,\eta \in \mathbb{R}^{n},\quad
p\geq 2,
\end{equation*}%
where $c_{p}$ is a positive constant depending only on $p,$ see \cite%
{Lindqvist}. As an immediate consequence of the above inequality with the
vectors $\xi =\varphi \nabla u$ and $\eta =u\nabla \varphi ,$ we obtain%
\begin{equation}
\left\vert \nabla \phi \right\vert ^{p}\geq \left\vert \nabla u\right\vert
^{p}\left\vert \varphi \right\vert ^{p}+u\left\vert \nabla u\right\vert
^{p-2}\nabla u\cdot \nabla \left( \left\vert \varphi \right\vert ^{p}\right)
+c_{p}\left\vert \nabla \varphi \right\vert ^{p}u^{p}.  \label{e4}
\end{equation}%
Multiplying the inequality $\left( \ref{e4}\right) $ by $a\left( x\right) $
on both sides, and then applying integration by parts over $\Omega $ for the
middle term, yields%
\begin{eqnarray*}
\int_{\Omega }a\left( x\right) \left\vert \nabla \phi \right\vert ^{p}dx
&\geq &\int_{\Omega }a\left( x\right) \left\vert \nabla u\right\vert
^{p}\left\vert \varphi \right\vert ^{p}dx-\int_{\Omega }\nabla \cdot \left(
a\left( x\right) u\left\vert \nabla u\right\vert ^{p-2}\nabla u\right)
\left\vert \varphi \right\vert ^{p}dx \\
&&+\int_{\partial \Omega }a\left( x\right) u\left\vert \nabla u\right\vert
^{p-2}\left\vert \varphi \right\vert ^{p}\partial _{\nu }ud\nu
+c_{p}\int_{\Omega }a\left( x\right) \left\vert \nabla \varphi \right\vert
^{p}u^{p}dx \\
&=&-\int_{\Omega }\nabla \cdot \left( a\left( x\right) \left\vert \nabla
u\right\vert ^{p-2}\nabla u\right) u\left\vert \varphi \right\vert ^{p}dx \\
&&+c_{p}\int_{\Omega }a\left( x\right) \left\vert \nabla \varphi \right\vert
^{p}u^{p}dx+\int_{\partial \Omega }a\left( x\right) u\left\vert \nabla
u\right\vert ^{p-2}\left\vert \varphi \right\vert ^{p}\partial _{\nu }ud\nu .
\end{eqnarray*}%
It therefore follows from $\left( \ref{de1}\right) $ and $\left( \ref{de2}%
\right) $ that%
\begin{equation*}
\int_{\Omega }a\left( x\right) \left\vert \nabla \phi \right\vert ^{p}dx\geq
\int_{\Omega }b\left( x\right) u^{p}\left\vert \varphi \right\vert
^{p}dx+c_{p}\int_{\Omega }a\left( x\right) \left\vert \nabla \varphi
\right\vert ^{p}u^{p}dx+\int_{\partial \Omega }\beta \left( x\right)
u^{p}\left\vert \varphi \right\vert ^{p}d\nu .
\end{equation*}%
Finally, making the variable change $\varphi =\frac{\phi }{u}$ in the above
integrals, one has the desired inequality for $p\geq 2:$%
\begin{equation*}
\int_{\Omega }a\left( x\right) \left\vert \nabla \phi \right\vert ^{p}dx\geq
\int_{\Omega }b\left( x\right) \left\vert \phi \right\vert
^{p}dx+c_{p}\int_{\Omega }a\left( x\right) \Big\vert \nabla \Big( \frac{
\phi }{u}\Big) \Big\vert ^{p}u^{p}dx+\int_{\partial \Omega }\beta \left(
x\right) \left\vert \phi \right\vert ^{p}d\nu .
\end{equation*}%
The inequality $\left( \ref{e3}\right) $ can be proved in the same spirit,
and in this case we apply the following convexity inequality:%
\begin{equation*}
\left\vert \xi +\eta \right\vert ^{p}\geq \left\vert \xi \right\vert
^{p}+p\left\vert \xi \right\vert ^{p-2}\xi \cdot \eta +c_{p}\frac{\left\vert
\eta \right\vert ^{2}}{\left( \left\vert \xi \right\vert +\left\vert \eta
\right\vert \right) ^{2-p}},\qquad \forall \xi ,\eta \in \mathbb{R}%
^{n},\quad 1<p<2,
\end{equation*}%
where $c_{p}=c\left( p\right) >0,$ see \cite{Lindqvist}. We shall omit the
details.
\end{proof}

\subsection{\textbf{Applications of Theorem} $\protect\ref{thm2.1}$}

As we stated earlier, our method is quite practical to construct various
weighted Hardy type inequalities including boundary terms on some bounded
domains in $%
%TCIMACRO{\U{211d} }%
%BeginExpansion
\mathbb{R}
%EndExpansion
^{n}$. To do this, we need to identify suitable model functions $a$ and $u$
that satisfy the above hypotheses in the given domain. Let us begin by
considering the model functions%
\begin{equation*}
a=\left\vert x\right\vert ^{\alpha }\text{\quad and\quad }u=\left\vert
x\right\vert ^{-\left( \frac{n+\alpha -p}{p}\right) }
\end{equation*}%
in Theorem $\ref{thm2.1}.$ After some computations, we readily get the
subsequent result:

\begin{corollary}
Let $\alpha \in 
%TCIMACRO{\U{211d} }%
%BeginExpansion
\mathbb{R}
%EndExpansion
$ and $n+\alpha >p>1.$ Then the following inequality holds:%
\begin{eqnarray*}
\int_{B_{R}}\left\vert x\right\vert ^{\alpha }\left\vert \nabla \phi
\right\vert ^{p}dx &\geq &\left( \frac{n+\alpha -p}{p}\right)
^{p}\int_{B_{R}}\left\vert x\right\vert ^{\alpha -p}|\phi |^{p}dx \\
&&-\left( \frac{n+\alpha -p}{p}\right) ^{p-1}R^{\alpha -p+1}\int_{\partial
B_{R}}|\phi |^{p}d\nu 
\end{eqnarray*}%
for all $\phi \in C^{\infty }(B_{R}).$
\end{corollary}

On the other hand, by applying Theorem $\ref{thm2.1}$ with the following pair%
\begin{equation*}
a=\left\vert x\right\vert ^{\alpha }\sinh ^{\gamma }\left\vert x\right\vert 
\text{\quad and\quad }u=\left\vert x\right\vert ^{-\left( \frac{n+\alpha
+\gamma -p}{p}\right) }
\end{equation*}%
and noting that $\left\vert x\right\vert \coth \left\vert x\right\vert \geq
1,$ we obtain the power hyperbolic sine $L^{p}$ Hardy-type inequality with a
boundary term.

\begin{corollary}
Let $\alpha \in 
%TCIMACRO{\U{211d} }%
%BeginExpansion
\mathbb{R}
%EndExpansion
,$ $\gamma \geq 0$ and $n+\alpha +\gamma >p>1.$ Then the following
inequality holds:%
\begin{eqnarray*}
\int_{B_{R}}\left\vert x\right\vert ^{\alpha }\sinh ^{\gamma }\left\vert
x\right\vert \left\vert \nabla \phi \right\vert ^{p}dx &\geq &\left( \frac{%
n+\alpha +\gamma -p}{p}\right) ^{p}\int_{B_{R}}\left\vert x\right\vert
^{\alpha -p}\sinh ^{\gamma }\left\vert x\right\vert |\phi |^{p}dx \\
&&-\left( \frac{n+\alpha +\gamma -p}{p}\right) ^{p-1}\frac{\sinh ^{\gamma }R%
}{R^{p-\alpha -1}}\int_{\partial B_{R}}|\phi |^{p}d\nu 
\end{eqnarray*}%
for all $\phi \in C^{\infty }(B_{R}).$
\end{corollary}

Recall that the following weighted $L^{2}$ Hardy type inequalities in the
Euclidean setting were proved by Ghoussoub and Moradifam \cite{Moradifam-1}:
Let $s,t>0$ and $\alpha ,\gamma ,m$ be real numbers.

\begin{itemize}
\item If $\alpha \gamma >0$ and $m\leq \frac{n-2}{2},$ then for all $\varphi
\in C_{0}^{\infty }(\mathbb{R}^{n})$ 
\begin{equation}
\int_{\mathbb{R}^{n}}\frac{\left( s+t\left\vert x\right\vert ^{\alpha
}\right) ^{\gamma }}{\left\vert x\right\vert ^{2m}}\left\vert \nabla \varphi
\right\vert ^{2}dx\geq \left( \frac{n-2m-2}{2}\right) ^{2}\int_{\mathbb{R}%
^{n}}\frac{\left( s+t\left\vert x\right\vert ^{\alpha }\right) ^{\gamma }}{%
\left\vert x\right\vert ^{2m+2}}\varphi ^{2}dx.  \label{m1}
\end{equation}

\item If $\alpha \gamma <0$ and $2m-\alpha \gamma \leq n-2,$ then for all $%
\varphi \in C_{0}^{\infty }(\mathbb{R}^{n})$ 
\begin{equation}
\int_{\mathbb{R}^{n}}\frac{\left( s+t\left\vert x\right\vert ^{\alpha
}\right) ^{\gamma }}{\left\vert x\right\vert ^{2m}}\left\vert \nabla \varphi
\right\vert ^{2}dx\geq \left( \frac{n+\alpha \beta -2m-2}{2}\right)
^{2}\int_{\mathbb{R}^{n}}\frac{\left( s+t\left\vert x\right\vert ^{\alpha
}\right) ^{\gamma }}{\left\vert x\right\vert ^{2m+2}}\varphi ^{2}dx.
\label{m2}
\end{equation}
\end{itemize}

We now extend and improve the above inequalities $\left( \ref{m1}\right) $
and $\left( \ref{m2}\right) $ to the $L^{p}$ case with a boundary term on
the $R-$ball in $\mathbb{R}^{n}$. In order to do so, we now take the pair as 
\begin{equation*}
a=\frac{\left( s+t\left\vert x\right\vert ^{\alpha }\right) ^{\gamma }}{%
\left\vert x\right\vert ^{pm}}\text{\quad and\quad }u=\left\vert
x\right\vert ^{-\left( \frac{n-pm-p}{p}\right) }
\end{equation*}%
in Theorem $\ref{thm2.1}.$ This gives the following improvement of the
inequality $\left( \ref{m1}\right) .$

\begin{corollary}
Let $s,t>0$ and $\alpha ,\gamma ,m\in \mathbb{R}$. If $\alpha \gamma >0$ and 
$1<p\leq n-pm,$ then for all $\phi \in C^{\infty }(B_{R})$ one has 
\begin{align*}
\int_{B_{R}}\frac{\left( s+t\left\vert x\right\vert ^{\alpha }\right)
^{\gamma }}{\left\vert x\right\vert ^{pm}}\left\vert \nabla \phi \right\vert
^{p}dx\geq & \ C_{n,p,m}^{p}\int_{B_{R}}\frac{\left( s+t\left\vert
x\right\vert ^{\alpha }\right) ^{\gamma }}{\left\vert x\right\vert ^{pm+p}}%
|\phi |^{p}dx \\
& +C_{n,p,m}^{p-1}\alpha \gamma t\int_{B_{R}}\frac{\left( s+t\left\vert
x\right\vert ^{\alpha }\right) ^{\gamma -1}}{\left\vert x\right\vert
^{pm+p-\alpha }}|\phi |^{p}dx \\
& -C_{n,p,m}^{p-1}\frac{\left( s+tR^{\alpha }\right) ^{\gamma }}{R^{pm+p-1}}%
\int_{\partial B_{R}}|\phi |^{p}d\nu ,
\end{align*}%
where $C_{n,p,m}=\left( \frac{n-pm-p}{p}\right) .$
\end{corollary}

If we consider the units 
\begin{equation*}
a=\frac{\left( s+t\left\vert x\right\vert ^{\alpha }\right) ^{\gamma }}{%
\left\vert x\right\vert ^{pm}}\text{\quad and\quad }u=\left\vert
x\right\vert ^{-\left( \frac{n+\alpha \gamma -pm-p}{p}\right) },
\end{equation*}%
then we immediately obtain the following improvement of the inequality $%
\left( \ref{m2}\right) .$

\begin{corollary}
Let $s,t>0$ and $\alpha ,\gamma ,m\in \mathbb{R}$. If $\alpha \gamma <0$ and 
$1<p\leq n+\alpha \gamma -pm,$ then for all $\phi \in C^{\infty }(B_{R})$
one has 
\begin{align*}
\int_{B_{R}}\frac{\left( s+t\left\vert x\right\vert ^{\alpha }\right)
^{\gamma }}{\left\vert x\right\vert ^{pm}}\left\vert \nabla \phi \right\vert
^{p}dx\geq & \ C_{n,p,m,\alpha ,\gamma }^{p}\int_{B_{R}}\frac{\left(
s+t\left\vert x\right\vert ^{\alpha }\right) ^{\gamma }}{\left\vert
x\right\vert ^{pm+p}}|\phi |^{p}dx \\
& -C_{n,p,m,\alpha ,\gamma }^{p-1}\alpha \gamma s\int_{B_{R}}\frac{\left(
s+t\left\vert x\right\vert ^{\alpha }\right) ^{\gamma -1}}{\left\vert
x\right\vert ^{pm+p}}|\phi |^{p}dx \\
& -C_{n,p,m,\alpha ,\gamma }^{p-1}\frac{\left( s+tR^{\alpha }\right)
^{\gamma }}{R^{pm+p-1}}\int_{\partial B_{R}}|\phi |^{p}d\nu ,
\end{align*}%
where $C_{n,p,m,\alpha ,\gamma }=\left( \frac{n+\alpha \gamma -pm-p}{p}%
\right) .$
\end{corollary}

\begin{remark}
Note that if $\alpha =0$ or $\gamma =0\ $in the above two inequalities, then
they reduce to Hardy type inequalities with usual weights. Hence, we are
interested in the case when $\alpha \gamma \neq 0.$
\end{remark}

It is worth stressing here that, by considering the model functions 
\begin{equation*}
a=\left( 1+\left\vert x\right\vert ^{\frac{p}{p-1}}\right) ^{\alpha \left(
p-1\right) }\text{ \quad and\quad }u=\left( 1+\left\vert x\right\vert ^{%
\frac{p}{p-1}}\right) ^{1-\alpha }
\end{equation*}%
in Theorem $\ref{thm2.1},$ we obtain an improved version of one of the
Skrzypczak's result in \cite{Skrzypczak} with a boundary term on the $R-$%
ball in $\mathbb{R}^{n}$.

\begin{corollary}
Let $1<p<n$ and $\alpha >1.$ Then for all $\phi \in C^{\infty }(B_{R}),$
there holds 
\begin{eqnarray*}
\int_{B_{R}}\left( 1+\left\vert x\right\vert ^{\frac{p}{p-1}}\right)
^{\alpha \left( p-1\right) }\left\vert \nabla \phi \right\vert ^{p}dx &\geq
&n\left( \frac{p\left( \alpha -1\right) }{p-1}\right)
^{p-1}\int_{B_{R}}\left( 1+\left\vert x\right\vert ^{\frac{p}{p-1}}\right)
^{\left( \alpha -1\right) \left( p-1\right) }|\phi |^{p}dx \\
&&-\left( \frac{p\left( \alpha -1\right) }{p-1}\right) ^{p-1}R\left( 1+R^{%
\frac{p}{p-1}}\right) ^{\left( \alpha -1\right) \left( p-1\right)
}\int_{\partial B_{R}}|\phi |^{p}d\nu .
\end{eqnarray*}
\end{corollary}

Another consequence of Theorem $\ref{thm2.1}$ with the functions 
\begin{equation*}
a=\left\vert x\right\vert ^{\alpha }\text{\quad and\quad }u=\left(
1+\left\vert x\right\vert ^{\frac{p}{p-1}}\right) ^{-\left( \frac{n+\alpha -p%
}{p}\right) }
\end{equation*}%
leads us to the following linearized Sobolev inequality:

\begin{corollary}
Let $\alpha \in \mathbb{R}$ and $n+\alpha >p>1.$ Then the inequality 
\begin{eqnarray*}
\int_{B_{R}}\left\vert x\right\vert ^{\alpha }\left\vert \nabla \phi
\right\vert ^{p}dx &\geq &\left( \frac{n+\alpha -p}{p-1}\right) ^{p-1}\left(
n+\alpha \right) \int_{B_{R}}\frac{\left\vert x\right\vert ^{\alpha }}{%
\left( 1+\left\vert x\right\vert ^{\frac{p}{p-1}}\right) ^{p}}|\phi |^{p}dx
\\
&&-\left( \frac{n+\alpha -p}{Rp-R}\right) ^{p-1}\int_{\partial B_{R}}|\phi
|^{p}d\nu 
\end{eqnarray*}%
holds for every $\phi \in C^{\infty }(B_{R}).$
\end{corollary}

Even though the literature has mostly focused on power radial weights, we
now establish $L^{p}$ Hardy-type inequalities with nonradial weights with
respect to the inclusion of boundary terms on smooth or piecewisely smooth
bounded domains in $%
%TCIMACRO{\U{211d} }%
%BeginExpansion
\mathbb{R}
%EndExpansion
^{n}$. In order to do so, we now take the pair as%
\begin{equation*}
a=\frac{e^{\alpha x_{1}}}{\alpha ^{p-1}}\text{\quad and\quad }u=e^{\alpha
x_{1}}.
\end{equation*}%
This gives the following result.

\begin{corollary}
\label{Cor2.1} Let $\alpha >0$ and $p>1.$ Then, for all $\phi \in C^{\infty
}\left( B_{R}\right) ,$ we have 
\begin{equation*}
\int_{B_{R}}\frac{e^{\alpha x_{1}}}{\alpha ^{p-1}}\left\vert \nabla \phi
\right\vert ^{p}dx\geq -\alpha p\int_{B_{R}}e^{\alpha x_{1}}|\phi |^{p}dx+ 
\frac{1}{R}\int_{\partial B_{R}}x_{1}e^{\alpha x_{1}}|\phi |^{p}d\nu .
\end{equation*}
\end{corollary}

\medskip

On the other hand, by specializing the functions as%
\begin{equation*}
a=e^{\alpha \left( 1-p\right) x_{1}}\text{\quad and\quad }u=e^{\alpha x_{1}},
\end{equation*}%
we have another weighted $L^{p}$ Hardy type inequality with a boundary term.

\begin{corollary}
\label{Cor2.2} Let $\alpha \in 
%TCIMACRO{\U{211d} }%
%BeginExpansion
\mathbb{R}
%EndExpansion
$ and $p>1.$ Then, for all $\phi \in C^{\infty }\left( B_{R}\right) ,$ we
have 
\begin{equation*}
\int_{B_{R}}e^{\alpha \left( 1-p\right) x_{1}}\left\vert \nabla \phi
\right\vert ^{p}dx\geq \frac{\alpha \left\vert \alpha \right\vert ^{p-2}}{R}
\int_{\partial B_{R}}x_{1}e^{\alpha \left( 1-p\right) x_{1}}|\phi |^{p}d\nu .
\end{equation*}
\end{corollary}

\medskip

We now set the non-symmetric functions%
\begin{equation*}
a=e^{-px_{1}}\text{\quad and\quad }u=e^{x_{1}}.
\end{equation*}%
This yields the following $L^{p}$ Hardy type inequality with non-symmetric
weights and a boundary remainder term.

\begin{corollary}
\label{Cor2.3} For any $\phi \in C^{\infty }\left( B_{R}\right) $ and $p>1,$
one has 
\begin{equation*}
\int_{B_{R}}\frac{\left\vert \nabla \phi \right\vert ^{p}}{e^{px_{1}}}dx\geq
\int_{B_{R}}\frac{|\phi |^{p}}{e^{px_{1}}}dx+\frac{1}{R}\int_{\partial
B_{R}} \frac{x_{1}}{e^{px_{1}}}|\phi |^{p}d\nu .
\end{equation*}
\end{corollary}

\medskip

Another immediate application of Theorem $\ref{thm2.1}$ with the following
functions%
\begin{equation*}
a\equiv 1\text{\quad and\quad }u=e^{x_{1}+\cdots +x_{n}}
\end{equation*}%
is the following inequality:

\begin{corollary}
\label{Cor2.4} For any $\phi \in C^{\infty }\left( B_{R}\right) $ and $p>1,$
one has 
\begin{equation*}
\int_{B_{R}}\left\vert \nabla \phi \right\vert ^{p}dx\geq \left( 1-p\right)
n^{p/2}\int_{B_{R}}|\phi |^{p}dx+\frac{n^{\frac{p-2}{2}}}{R}\int_{\partial
B_{R}}\left( x_{1}+\cdots +x_{n}\right) |\phi |^{p}d\nu .
\end{equation*}
\end{corollary}

\medskip

Let us mention that when one consider Theorem $\ref{thm2.1}$ with the pair%
\begin{equation*}
a=e^{x_{1}+\cdots +x_{n}}\text{\quad and\quad }u=e^{x_{1}+\cdots +x_{n}},
\end{equation*}%
one can obtain the following result including an exponential weight:

\begin{corollary}
\label{Cor2.5} For any $\phi \in C^{\infty }\left( B_{R}\right) $ and $p>1,$
one has 
\begin{eqnarray*}
\int_{B_{R}}e^{x_{1}+\cdots +x_{n}}\left\vert \nabla \phi \right\vert ^{p}dx
&\geq &-pn^{p/2}\int_{B_{R}}e^{x_{1}+\cdots +x_{n}}|\phi |^{p}dx \\
&&+\frac{n^{\frac{p-2}{2}}}{R}\int_{\partial B_{R}}\left( x_{1}+\cdots
+x_{n}\right) e^{x_{1}+\cdots +x_{n}}|\phi |^{p}d\nu .
\end{eqnarray*}
\end{corollary}

\medskip

\begin{remark}
\bigskip Note that one can also apply Theorem $\ref{thm2.1}$ on some domains 
$\Omega $ in $%
%TCIMACRO{\U{211d} }%
%BeginExpansion
\mathbb{R}
%EndExpansion
^{n}$ with piecewisely smooth boundary $\partial \Omega .$ The following is
the first corollary in this direction.
\end{remark}

Suppose that%
\begin{equation*}
a\equiv 1\text{\quad and\quad }u=x_{1}r,
\end{equation*}%
then we derive the following inequality:

\begin{corollary}
\label{Cor2.6} Let $\Omega =\left\{ x\in B_{R}:x_{1}>0\right\} .$ For all $%
\phi \in C^{\infty }(\Omega ),$ there holds 
\begin{equation*}
\int_{\Omega }\left\vert \nabla \phi \right\vert ^{2}dx\geq -\left(
n+1\right) \int_{\Omega }\frac{\phi ^{2}}{\left\vert x\right\vert ^{2}}%
dx+\int_{\partial \Omega }\left( \frac{1}{R}-\frac{1}{x_{1}}\right) \phi
^{2}d\nu .
\end{equation*}
\end{corollary}

\medskip

When we take the following pair%
\begin{equation*}
a=\left( \frac{1}{x_{1}^{2}}+\frac{1}{x_{2}^{2}}+\cdots +\frac{1}{x_{n}^{2}}%
\right)^{\frac{2-p}{p}} \text{\quad and\quad }u=\left( x_{1}x_{2}\ldots
x_{n}\right) ^{1/p}
\end{equation*}%
in Theorem $\ref{thm2.1},$ we have the subsequent weighted $L^{p}$ Hardy
type inequality with boundary terms.

\begin{corollary}
\label{Cor2.7} Let $1<p<\infty \ $and $\Omega =\left\{ x\in B_{R}:x_{i}>0,%
\text{ } i=1,2,\ldots ,n\right\} .$ Then, for all $\phi \in C^{\infty
}\left( \Omega \right) ,$ we have 
\begin{eqnarray*}
\int_{\Omega }\left( \frac{1}{x_{1}^{2}}+\frac{1}{x_{2}^{2}}+\cdots +\frac{1 
}{x_{n}^{2}}\right) ^{\frac{2-p}{2}}\left\vert \nabla \phi \right\vert
^{p}dx &\geq &\frac{1}{p^{p}}\int_{\Omega }\left( \frac{1}{x_{1}^{2}}+\frac{%
1 }{x_{2}^{2}}+\cdots +\frac{1}{x_{n}^{2}}\right) |\phi |^{p}dx \\
&&-\frac{1}{p^{p-1}}\int_{\partial \Omega }\left( \frac{1}{x_{1}^{2}}+\frac{%
1 }{x_{2}^{2}}+\cdots +\frac{1}{x_{n}^{2}}\right) |\phi |^{p}d\nu \\
&&+\frac{n}{Rp^{p-1}}\int_{\partial \Omega }|\phi |^{p}d\nu .
\end{eqnarray*}
\end{corollary}

\medskip

We now set the pair as%
\begin{equation*}
a=x_{1}^{p-1}\text{\quad and\quad }u=\left( \log \frac{R}{x_{1}}\right) ^{ 
\frac{p-1}{p}}.
\end{equation*}%
Hence, we derive the power logarithmic $L^{p}$ Hardy type inequality.

\begin{corollary}
\label{Cor2.8} Let $p>1$ and $\Omega =\left\{ x\in B_{R}:x_{1}>0\right\} .$
Then, for all $\phi \in C^{\infty }\left( \Omega \right) ,$ we have 
\begin{eqnarray*}
\int_{\Omega }x_{1}^{p-1}\left\vert \nabla \phi \right\vert ^{p}dx &\geq
&\left( \frac{p-1}{p}\right) ^{p}\int_{\Omega }\frac{|\phi |^{p}}{%
x_{1}\left( \log \frac{R}{x_{1}}\right) ^{p}}dx \\
&&+\left( \frac{p-1}{p}\right) ^{p-1}\frac{1}{R}\int_{\partial \Omega }\frac{%
R-x_{1}}{\left( \log \frac{R}{x_{1}}\right) ^{p-1}}|\phi |^{p}d\nu .
\end{eqnarray*}
\end{corollary}

\medskip

Inspired by the inequality of Maz'ya \eqref{1.2} and considering the
functions

\begin{equation*}
a=x_{n}^{p-1}\text{\quad and\quad }u=\left(x_{n-1}^2+x_n^2\right) ^{ -\frac{1%
}{2p}},
\end{equation*}%
we obtain the following Maz'ya type inequality with a boundary term.

\begin{corollary}
\label{Cor2.9} Let $p>1$ and $\Omega =\left\{ x\in B_{R}:x_{n}>0\right\} .$
Then, for all $\phi \in C^{\infty }\left( \Omega \right) ,$ we have 
\begin{eqnarray*}
\int_{\Omega }x_{n}^{p-1}\left\vert \nabla \phi \right\vert ^{p}dx &\geq& 
\frac{ 1}{p^{p}}\int_{\Omega }\frac{|\phi |^{p}}{( x_{n-1}^{2}+x_{n}^{2})
^{p/2}}dx \\
&&-\frac{1}{Rp^{p-1}}\int_{\partial \Omega }\frac{x_{n}^{p-1}(
x_{n-1}^{2}+x_{n}^{2}-Rx_{n}) }{\left( x_{n-1}^{2}+x_{n}^{2}\right) ^{\frac{p%
}{2}}}|\phi |^{p}d\nu .
\end{eqnarray*}
\end{corollary}

\medskip

Another consequence of the Theorem $\ref{thm2.1}$ with the special functions

\begin{equation*}
a=e^{\alpha x_{1}^{2}x_2^2\dots x_{n}^{2}}\text{\quad and\quad }u=e^{-\alpha
x_{1}^{2}x_2^2 \dots x_{n}^{2}}
\end{equation*}%
is the following inequality.

\begin{corollary}
\label{Cor2.10} For all $\phi \in C^{\infty }\left( B_{R}\right) $ and $%
\alpha >0,$ we have 
\begin{equation*}
\begin{aligned} \int_{B_{R}}e^{\alpha x_{1}^{2}x_2^2\dots
x_{n}^{2}}\left\vert \nabla \phi \right\vert ^{2}dx\geq & 2\alpha
\int_{B_{R}} e^{\alpha x_{1}^{2}x_2^2\dots
x_{n}^{2}}\left(x_{1}^{2}x_2^2\dots
x_{n}^{2}\right)\left(\frac{1}{x_{1}^{2}}+\cdots +\frac{1}{x_{n}^{2}}\right)
|\phi|^{2}dx\\ &-\frac{2\alpha n}{R}\int_{\partial B_{R}}e^{\alpha
x_{1}^{2}x_2^2\dots x_{n}^{2}}\left(x_{1}^{2} \dots
x_{n}^{2}\right)|\phi|^{2}d\nu . \end{aligned}
\end{equation*}
\end{corollary}

\medskip

\subsection*{Heisenberg-Pauli-Weyl Inequality}

Finally, we turn our attention to the classical Heisenberg-Pauli-Weyl
inequality which states that 
\begin{equation}
\left( \int_{\mathbb{R}^{n}}|\nabla f(x)|^{2}dx\right) \left( \int_{\mathbb{R%
}^{n}}|x|^{2}|f(x)|^{2}dx\right) \geq \frac{n^{2}}{4}\left( \int_{\mathbb{R}%
^{n}}|f(x)|^{2}dx\right) ^{2}  \label{HPW}
\end{equation}%
for all $f\in L^{2}(\mathbb{R}^{n})$. It has been studied in various
contexts and we refer to \cite{Folland-Sitaram} and \cite{Erb} for an
overview of the history and the relevance of this inequality.

\bigskip It is worth mentioning here that Theorem $\ref{thm2.1}$ can be
applied to obtain the local Heisenberg-Pauli-Weyl-type inequalities with
boundary terms. In order to be more precise, we now take $p=2$ in Theorem $%
\ref{thm2.1}$ and choose the following functions%
\begin{equation*}
a\equiv 1\text{\quad and\quad }u=e^{-\alpha \left\vert x\right\vert
^{2}},\quad \alpha >0.
\end{equation*}%
This allows the derivation of the inequality%
\begin{equation*}
\int_{B_{R}}\left\vert \nabla \phi \right\vert ^{2}dx\geq -4\alpha
^{2}\int_{B_{R}}\left\vert x\right\vert ^{2}\phi ^{2}dx+2\alpha
n\int_{B_{R}}\phi ^{2}dx-2\alpha R^{2}\int_{\partial B_{R}}\phi ^{2}d\nu .
\end{equation*}%
Optimizing in $\alpha $, we find the local Heisenberg-Pauli-Weyl type
inequality with a boundary term.

\begin{corollary}
For all $\phi \in C^{\infty }\left( B_{R}\right) ,$ we have%
\begin{equation*}
\left( \int_{B_{R}}\left\vert \nabla \phi \right\vert ^{2}dx\right) \left(
\int_{B_{R}}\left\vert x\right\vert ^{2}\phi ^{2}dx\right) \geq \frac{1}{4}%
\left( n\int_{B_{R}}\phi ^{2}dx-R^{2}\int_{\partial B_{R}}\phi ^{2}d\nu
\right) ^{2}.
\end{equation*}
\end{corollary}

On the other hand, by setting the functions 
\begin{equation*}
a\equiv 1\text{\quad and\quad }u=e^{-\alpha \left\vert x\right\vert },\quad
\alpha >0,
\end{equation*}%
we shall deduce from Theorem $\ref{thm2.1}\ $that 
\begin{equation*}
\int_{B_{R}}\left\vert \nabla \phi \right\vert ^{2}dx\geq -\alpha
^{2}\int_{B_{R}}\phi ^{2}dx+\alpha \left( \left( n-1\right) \int_{B_{R}}%
\frac{\phi ^{2}}{\left\vert x\right\vert }dx-\int_{\partial B_{R}}\phi
^{2}d\nu \right) .
\end{equation*}%
Hence, a very similar argument applies to obtain the following version of
the Heisenberg-Pauli-Weyl type inequality.

\begin{corollary}
For every $\phi \in C^{\infty }(B_{R}),$ one has 
\begin{equation*}
\left( \int_{B_{R}}\left\vert \nabla \phi \right\vert ^{2}dx\right) \left(
\int_{B_{R}}\phi ^{2}dx\right) \geq \frac{1}{4}\left( \left( n-1\right)
\int_{B_{R}}\frac{\phi ^{2}}{\left\vert x\right\vert }dx-\int_{\partial
B_{R}}\phi ^{2}d\nu \right) ^{2}.
\end{equation*}
\end{corollary}

Considering the functions

\begin{equation*}
a\equiv 1\text{\quad and\quad }u=e^{-\alpha x_{n}^{2}},\quad \alpha >0,
\end{equation*}%
as a consequence of Theorem $\ref{thm2.1}$, we readily have 
\begin{equation*}
\int_{B_{R}}\left\vert \nabla \phi \right\vert ^{2}dx\geq -4\alpha
^{2}\int_{B_{R}}x_{n}^{2}|\phi |^{2}dx+2\alpha \left( \int_{B_{R}}|\phi
|^{2}dx-\int_{\partial B_{R}}x_{n}^{2}|\phi |^{2}d\nu \right) .
\end{equation*}%
Arguing as above, we obtain the following inequality.

\begin{corollary}
\label{Cor2.11} For every $\phi \in C^{\infty }(B_{R}),$ one has 
\begin{equation*}
\left( \int_{B_{R}}\left\vert \nabla \phi \right\vert ^{2}dx\right) \left(
\int_{B_{R}}x_{n}^{2}|\phi |^{2}dx\right) \geq \frac{1}{4}\left(
\int_{B_{R}}|\phi |^{2}dx-\int_{\partial B_{R}}x_{n}^{2}|\phi |^{2}d\nu
\right) ^{2}.
\end{equation*}
\end{corollary}

\medskip

Finally, when we take the pair as

\begin{equation*}
a\equiv 1\text{\quad and\quad }u=e^{-\alpha \left(
x_{n-1}^{2}+x_{n}^{2}\right) },\quad \alpha >0,
\end{equation*}%
we get%
\begin{equation*}
\begin{aligned} \int_{B_{R}}\left\vert \nabla \phi \right\vert ^{2}dx\geq&
-4\alpha ^{2}\int_{B_{R}}\left( x_{n-1}^{2}+x_{n}^{2}\right)
|\phi|^{2}dx+4\alpha \int_{B_{R}}|\phi|^{2}dx\\ &-\frac{2\alpha
}{R}\int_{\partial B_{R}}\left( x_{n-1}^{2}+x_{n}^{2}\right) |\phi|^{2}d\nu
. \end{aligned}
\end{equation*}%
Hence, with a similar approach, we achieve the following inequality.

\begin{corollary}
\label{Cor2.12} For all $\phi \in C^{\infty }\left( B_{R}\right) ,$ we have 
\begin{equation*}
\left( \int_{B_{R}}\left\vert \nabla \phi \right\vert ^{2}dx\right) \left(
\int_{B_{R}}\left( x_{n-1}^{2}+x_{n}^{2}\right) |\phi|^{2}dx\right) \geq 
\frac{A}{4},
\end{equation*}
where $A=\left( 2\int_{B_{R}}|\phi|^{2}dx-\frac{1}{R}\int_{\partial
B_{R}}\left( x_{n-1}^{2}+x_{n}^{2}\right) |\phi|^{2}d\nu \right) ^{2}$.
\end{corollary}

\end{document}